\documentclass[preprint,12pt]{elsarticle}
\usepackage{amsmath,amssymb,amsthm,color,url}

\newtheorem{thr}{Theorem}

\newtheorem{claim}[thr]{Claim}

\theoremstyle{definition}

\theoremstyle{remark}

\journal{arxiv.org}

\begin{document}

\begin{frontmatter}

\title{Linear mappings preserving the copositive cone}

\author{Yaroslav Shitov}

\ead{yaroslav-shitov@yandex.ru}


\begin{abstract}
Let $\mathcal{S}_n$ be the set of all $n$-by-$n$ symmetric real matrices, and let $\mathcal{C}_n$ be the \textit{copositive} cone, that is, the set of all matrices $a\in\mathcal{S}_n$ that fulfill the condition $u^\top a u\geqslant0$ for all $n$-vectors $u$ with nonnegative entries. We prove that a linear mapping $\varphi:\mathcal{S}_n\to \mathcal{S}_n$ satisfies $\varphi(\mathcal{C}_n)=\mathcal{C}_n$ if and only if
$$\varphi(x)=m^\top xm$$
for a fixed monomial matrix $m$ with nonnegative entries.
\end{abstract}

\end{frontmatter}

\smallskip

Our paper is a contribution to the study of \textit{linear preservers}, which aims to describe linear operators on matrix algebras with respect to either a given property of matrices or according to a function defined on them. As suggested in modern literature~\cite{LT}, this study dates back to a 1897 paper~\cite{Frob} by Frobenius, who described linear operators $\varphi$ on $n\times n$ complex matrices that satisfy $\det \varphi(x)=\det x$ for all $x$. This study remains attractive for many researchers in linear algebra, as shown by the work discussed in the surveys~\cite{GLS, LP, LT} and their citation counts. This paper continues the investigation of~\cite{FJZ} and applies this study to another well known concept of linear algebra---we work with preserver problems on the \textit{copositive} cone, an object that is being studied since at least as early as the 60's~\cite{MM}. Besides from being an attractive object to study in its own right~\cite{BDS}, the copositive cone is relevant in different contexts of optimization~\cite{DG2} and combinatorics~\cite{HallBook}. 

\section{Our result}

In this section and the rest of our note, we deal with real matrices and vectors. They are called \textit{nonnegative} if all numbers they are formed of are nonnegative. The vectors are assumed to be written as columns, and we write $x\geqslant 0$ if a vector $x$ is nonnegative. An $n\times n$ matrix is called \textit{monomial} if it has exactly one non-zero entry in every row and in every column. As said in the abstract, the notation $\mathcal{S}_n$ stands for the set of all $n$-by-$n$ symmetric real matrices, and $\mathcal{C}_n$ is the set of \textit{copositive} matrices, which are those elements $a\in\mathcal{S}_n$ that satisfy the inequality $u^\top a u\geqslant0$ for all nonnegative $n$-vectors $u$. A linear operator defined on $\mathcal{S}_n$ as $x\to m^\top xm$, for a fixed nonnegative monomial matrix $m$, is called a \textit{monomial congruence}.

\begin{thr}\label{thrmcopos}
A linear mapping $\varphi:\mathcal{S}_n\to \mathcal{S}_n$ satisfies $\varphi(\mathcal{C}_n)=\mathcal{C}_n$ if and only if $\varphi$ is a monomial congruence.
\end{thr}

This result was conjectured in a recent paper~\cite{FJZ} of Furtado, Johnson, Zhang, who checked its validity with $n=2$, and also for mappings of the form $x\to s^\top xs$ with a fixed, not necessarily monomial, matrix $s$. The 'if' part of Theorem~\ref{thrmcopos} is trivial, so we concentrate on the 'only if' direction.

\section{Auxiliary results}

In the rest of our paper, the letter $\varphi$ stands for a mapping as in Theorem~\ref{thrmcopos}, that is, a linear operator on $\mathcal{S}_n$ that satisfies $\varphi(\mathcal{C}_n)=\mathcal{C}_n$. We work with the standard Euclidean topology on the space of all $n\times n$ real matrices, and we write $\partial\mathcal{C}_n$ to denote the boundary of $\mathcal{C}_n$ in $\mathcal{S}_n$. We list several known propositions and give brief sketches of their proofs for completeness.

\begin{claim}\label{cl1copos}{\upshape (See~\cite[Theorem~3.1]{FJZ}.)}
The mapping $\varphi$ is bijective on $\mathcal{S}_n$, and its inverse $\varphi^{-1}$ satisfies $\varphi^{-1}(\mathcal{C}_n)=\mathcal{C}_n$.
\end{claim}

\begin{proof}
The first part follows because $\mathcal{C}_n$ spans $\mathcal{S}_n$ over $\mathbb{R}$, and the assertion with the inverse follows from $\varphi(\mathcal{C}_n)=\mathcal{C}_n$ by applying $\varphi^{-1}$ to both sides.
\end{proof}

\begin{claim}\label{cl2copos}{\upshape (See~\cite[Corollary~3.2]{FJZ}.)}
We have $\varphi(\partial\mathcal{C}_n)=\partial\mathcal{C}_n$.
\end{claim}

\begin{proof}
Follows because $\varphi$ is bijective and continuous.
\end{proof}


\begin{claim}\label{cl3copos}{\upshape (See~\cite[page 595, Remark]{HUS}.)}
If $a\in\partial\mathcal{C}_n$, then there is a non-zero vector $\xi\geqslant 0$ such that $\xi^\top a\xi=0$.
\end{claim}

\begin{proof}
Pick a vector that minimizes the value $x^\top ax$ over the compact set of all vectors $x=(x_1,\ldots,x_n)$ such that $x\geqslant 0$ and $x_1+\ldots+x_n=1$.
\end{proof}


\begin{claim}\label{cl4copos}{\upshape (See~\cite[Theorem~3.2]{Val}.)}
If $a\in\mathcal{C}_n$ and $\xi^\top a\xi=0$ for an $n$-vector $\xi$ with all coordinates strictly positive, then $a\xi$ is a zero vector.
\end{claim}

\begin{proof}
Since any vector $x$ that lies sufficiently close to $\xi$ is nonnegative, it has to satisfy $x^\top a x\geqslant 0$. This means that the quadratic form $A(x)=x^\top a x$ attains its local minimum at $\xi$, which shows that the gradient $\partial A/\partial x$ should be zero at $\xi$. It remains to note that this gradient equals $2ax$.
\end{proof}

\section{The proof of Theorem~\ref{thrmcopos}}

For a fixed index $t\in\{1,\ldots, n\}$, we define $\mathcal{A}^t$ as the set of all matrices in $\mathcal{S}_n$ with zeros at the $(t,t)$ position and positive numbers everywhere else. 

\begin{claim}\label{cl5copos}
We have $\mathcal{A}^t\subset\partial\mathcal{C}_n$ and  $\varphi(\mathcal{A}^t)\subset\partial\mathcal{C}_n$.
\end{claim} 

\begin{proof}
The first inclusion follows because copositive matrices have nonnegative diagonal, and the second one is guaranteed by Claim~\ref{cl2copos}.
\end{proof}

\begin{claim}\label{cl6copos}
The image $\varphi(\mathcal{A}^t)$ spans a codimension-one subspace of $\mathcal{S}_n$.
\end{claim} 

\begin{proof}
Follows because $\varphi$ is bijective.
\end{proof}

In what follows, we associate any copositive matrix $b$ with the kernel $\ker b=\{x\geqslant 0: x^\top b x=0\}$ of the corresponding quadratic form.

\begin{claim}\label{cl7copos}
We have $\ker b=\ker c$ for all $b,c\in\varphi(\mathcal{A}^t)$.
\end{claim}

\begin{proof}
Since $\varphi^{-1}(b),\varphi^{-1}(c)$ belong to $\mathcal{A}^t$, we can find $a\in\mathcal{A}^t$ and $\varepsilon>0$ satisfying $\varphi^{-1}(b)=\varepsilon \varphi^{-1}(c)+a$, which shows that $b=\varepsilon c+\varphi(a)$. Since $\varphi(a)$ is copositive, the latter equation shows that $\ker b\subseteq\ker c$.
\end{proof}

\begin{claim}\label{cl8copos}
There is a nonnegative vector $u^t$ with exactly one non-zero coordinate such that $\ker b=\{\lambda u^t|\lambda\in\mathbb{R},\lambda\geqslant 0\}$ for all $b\in\varphi(\mathcal{A}^t)$.
\end{claim}

\begin{proof}
We find a non-zero vector $u\in \ker b$ by Claim~\ref{cl3copos}, and we define $\sigma$ as the set of all non-zero indexes of $u$. Since the principal submatrix $b(\sigma|\sigma)$ is still copositive, we can apply Claim~\ref{cl4copos} and get $b(\sigma|\sigma)u'=0$, where $u'$ is the vector obtained from $u$ by removing its zero coordinates. According to Claim~\ref{cl7copos}, the vector $u$ belongs to the kernel of any $c\in\varphi(\mathcal{A}^t)$, so the condition $c(\sigma|\sigma)u'=0$ is satisfied by all such $c$. Since this condition is a union of $|\sigma|$ independent linear equations, we get $|\sigma|\leqslant 1$ from Claim~\ref{cl6copos}. The existence of two non-collinear vectors in $\ker b$ contradicts to Claim~\ref{cl6copos} as well.
\end{proof}

In what follows, we write $e_{ij}$ to denote the $(i,j)$ matrix unit. We define $\mathcal{T}_n$ as the set of nonnegative matrices with zeros at the diagonal, which are exactly the set of trace-zero copositive matrices.

\begin{claim}\label{cl9copos}
There is a monomial congruence $\mu$ such that the mapping $\psi=\varphi\circ\mu$ satisfies $\psi(e_{tt})=e_{tt}+b_t$ with $b_t\in\mathcal{T}_n$, for all indexes $t$.
\end{claim}

\begin{proof}
Let $\pi(t)$ be the index of the non-zero coordinate of a vector $u^t$ as in Claim~\ref{cl8copos}. If we had $\pi(s)=\pi(t)=j$ for some $s\neq t$, all the matrices in $\varphi\left(\mathcal{A}^s+\mathcal{A}^{t}\right)$ would have zeros at the $(j,j)$ position, which is a contradiction because the linear span of $\mathcal{A}^s+\mathcal{A}^{t}$ is full-dimensional. So we see that $\pi$ is an injective mapping, and as such it is a permutation of $\{1,\ldots,n\}$. For any $k\neq j$, we have $k=\pi(q)$ with $q\neq t$, so the $(k,k)$ entry of $\varphi(e_{tt})$ should be zero because $\varphi(e_{tt})$ belongs to the closure of $\varphi(\mathcal{A}^{q})$. 
So we have $\varphi(e_{tt})=\alpha_t e_{jj}+b_t$ with a real number $\alpha_t$ and a
trace-zero matrix $b_t$, and then $b_t\in\mathcal{T}_n$ because $\varphi(e_{tt})$ is copositive. Finally, since the set $\mathcal{T}_n$ lies in the closure of $\mathcal{A}^1\cap\ldots\cap\mathcal{A}^{n}$, we have $\varphi(\mathcal{T}_n)\subseteq\mathcal{T}_n$ and then $\varphi^{-1}(\mathcal{T}_n)\subseteq\mathcal{T}_n$. This shows that $\alpha_t\neq 0$ and gives a desired result up to a monomial congruence.
\end{proof}

\begin{claim}\label{cl10copos}
For $\psi$ as in Claim~\ref{cl9copos}, we have $\psi(x)=x$.
\end{claim}

\begin{proof}
We get $$\psi^{-1}(e_{tt}+0.5b_t)=\psi^{-1}(e_{tt}+b_t)-0.5\psi^{-1}(b_t)=e_{tt}-0.5\psi^{-1}(b_t),$$ and since the resulting matrix is copositive, the preimage $\psi^{-1}(b_t)$ should be collinear to $e_{tt}$. Taking the $\psi$, we see that $b_t$ is collinear to $\psi(e_{tt})=e_{tt}+b_t$, so we have $b_t=0$. This result means that $\psi(e_{tt})=e_{tt}$.

Finally, we write $\beta_{ij}=\psi(e_{ij}+e_{ji})$ for $i\neq j$. As noted in the proof of the previous claim, $\beta_{ij}$ is trace-zero. The matrix $\psi(e_{ii}+e_{jj}-e_{ij}-e_{ji})=e_{ii}+e_{jj}-\beta_{ij}$ is copositive, so the non-zero entries of $\beta_{ij}$ lie in the $2\times 2$ submatrix with indexes $i,j$. This means that $\beta_{ij}=\alpha_{ij}(e_{ij}+e_{ji})$ with some $\alpha_{ij}\in\mathbb{R}$, and we should have $\alpha_{ij}\leqslant 1$ again because $e_{ii}+e_{jj}-\beta_{ij}$ is copositive. The opposite inequality $\alpha_{ij}\geqslant 1$ follows by symmetry.
\end{proof}

Claims~\ref{cl9copos} and~\ref{cl10copos} show that $\varphi$ is an identical mapping up to a monomial congruence, so the proof of Theorem~\ref{thrmcopos} is complete.



\begin{thebibliography}{10}

\bibitem{BDS}
A. Berman, M. D\"{u}r, N. Shaked-Monderer, Open problems in the theory of completely positive and copositive matrices, \textit{Electron. J. Linear Al.} 29 (2015) 46--58.

\bibitem{DG2}
P. J. C. Dickinson, L. Gijben, On the computational complexity of membership problems for the completely positive cone and its dual, \textit{Comput. Optim. Appl.} 57 (2014) 403--415.

\bibitem{Frob}
G. Frobenius, \"{U}ber die Darstellung der endlichen Gruppen durch lineare Substitutionen, \textit{Sitzungsber. Deutsch. Akad. Wiss. Berlin} (1897) 994--1015.

\bibitem{FJZ}
S. Furtado, C. R. Johnson, Y. Zhang, Linear preservers of copositive matrices, \textit{Linear Multilinear A.} (2019) \url{doi.org/10.1080/03081087.2019.1692775}.

\bibitem{GLS}
A. Guterman, C. K. Li, P. \v{S}emrl, Some general techniques on linear preserver problems, \textit{Linear Algebra Appl.} 315 (2000) 61--81.

\bibitem{HallBook}
M. Hall Jr. \textit{Combinatorial Theory}. Blaisdell, Lexington, 1967.

\bibitem{HUS}
J.-B. Hiriart-Urruty, A. Seeger, A variational approach to copositive matrices, \textit{SIAM Review} 52 (2010) 593--629.

\bibitem{LP}
C. K. Li, S. Pierce, Linear preserver problems, \textit{Am. Math. Mon.} 108 (2001) 591--605.

\bibitem{LT}
C. K. Li, N. K. Tsing, Linear preserver problems: A brief introduction and some special techniques, \textit{Linear Algebra Appl.} 162 (1992) 217--235.

\bibitem{MM}
J. E. Maxfield, H. Minc, On the matrix equation $X'X= A$, \textit{P. Edinburgh Math. Soc.} 13 (1962) 125--129.

\bibitem{Val}
H. V\"{a}liaho, Criteria for copositive matrices, \textit{Linear Algebra Appl.} 81 (1986) 19--34.

\end{thebibliography}
\end{document}